\documentclass{amsart}
\usepackage[all,2cell]{xy}
\usepackage{ifpdf}
\usepackage{amsmath, amsthm, amsfonts, amssymb, amscd}
\usepackage[active]{srcltx}
\usepackage[pdftex]{graphicx}
\usepackage{color} 
    \ifpdf

      \usepackage[pdftex]{graphicx}  

      \usepackage[pdftex]{hyperref}
      \makeatletter
	\let\Hy@linktoc\Hy@linktoc@none
	\makeatother

    \else

      \usepackage[dvips]{graphicx}  

      \newcommand{\href}[2]{#2}

    \fi
\newcommand{\xxx}[1]{{{\color{red}{\emph{#1}}}}}

\newcommand{\ine}{\operatorname{Ine}}
\newcommand{\ess}{\operatorname{Ess}}

\newcommand{\abs}[1]{\left\lvert{#1}\right\rvert}
\newcommand{\norm}[1]{\left\|{#1}\right\|}

\DeclareMathOperator{\GL}{\rm{GL}}

\DeclareMathOperator{\bd}{\partial}

\DeclareMathOperator{\fix}{\rm{Fix}}

\newcommand{\ol}{\overline}
\renewcommand{\hat}{\widehat}
\newcommand{\til}{\widetilde}

\newcommand{\R}{\mathbb{R}}\newcommand{\N}{\mathbb{N}}
\newcommand{\Z}{\mathbb{Z}}\newcommand{\Q}{\mathbb{Q}}
\newcommand{\T}{\mathbb{T}}
\newcommand{\A}{\mathbb{A}}

\newcommand{\sm}{\setminus}

\newcommand{\id}{\mathrm{Id}}
\newcommand{\deck}{\operatorname{Deck}}

\newcommand{\ie}{i.e.\ }
\newcommand{\eg}{e.g.\ }

\newtheorem{theorem}{Theorem}[section]
\newtheorem{corollary}[theorem]{Corollary}
\newtheorem{lemma}[theorem]{Lemma}
\newtheorem{proposition}[theorem]{Proposition}

\newtheorem{claim}{Claim}
\newtheorem*{theorem*}{Theorem}

\newtheorem{theoremain}{Theorem}

\newtheorem{corollarymain}[theoremain]{Corollary}

\theoremstyle{definition}

\theoremstyle{remark}
\newtheorem{remark}[theorem]{Remark}

\title[A characterization of annularity for area-preserving 
]{A characterization of annularity for area-preserving toral homeomorphisms}
\author{Nancy Guelman}
\address{Nancy Guelman. IMERL, Facultad de Ingenier\'\i a, Universidad de la Rep\'ublica, C.C. 30, Montevideo, Uruguay}
\email{nguelman@fing.edu.uy}
\author{Andres Koropecki}
\address{Andres Koropecki. Universidade Federal Fluminense, Instituto de Matem\'atica e Estat\'\i stica, Rua M\'ario Santos Braga S/N, 24020-140 Niteroi, RJ, Brasil}
\email{ak@id.uff.br}
\author{Fabio Armando Tal}
\address{Fabio Armando Tal. Instituto de Matem\'atica e Estat\'\i stica, Universidade de S\~ao Paulo, Rua do Mat\~ao 1010, Cidade Universit\'aria, 05508-090 S\~ao Paulo, SP, Brazil}
\email{fabiotal@ime.usp.br}
\thanks{The first author was partially supported by CNPq-Brasil. The second author was partially supported by FAPESP and CNPq-Brasil}

\begin{document}

\begin{abstract} We prove that if an area-preserving homeomorphism of the torus in the homotopy class of the identity has a rotation set which is a nondegenerate vertical segment containing the origin, then there exists an essential invariant annulus. In particular, some lift to the universal covering has uniformly bounded displacement in the horizontal direction.
\end{abstract}

\maketitle



\section{Introduction}
Let $\T^2 = \R^2/\Z^2$ denote the two-dimensional torus, and $\pi\colon \R^2\to \Z^2$ the canonical projection. Consider a homeomorphism $f\colon \T^2\to \T^2$ homotopic to the identity, and a lift $\hat{f}\colon \R^2\to \R^2$ of $f$. The \emph{rotation set} $\rho(\hat{f})$ of $\hat{f}$, introduced by Misiurewicz and Ziemian in \cite{m-z} as a generalization of the rotation number of an orientation-preserving circle homeomorphism, is defined as the set of all limits of sequences of the form
\begin{equation} v = \lim_{k\to \infty} (\hat{f}^{n_k}(z_k)-z_k)/n_k,\end{equation}
where $(n_k)_{k\in \N}$ is a sequence of integers such that $n_k\to \infty$ as $k\to \infty$.
 Roughly speaking, this set measures the average asymptotic rotation of orbits. It is known that $\rho(\hat{f})$ is always compact and convex, and if $v\in \rho(\hat{f})$, is extremal or interior in $\rho(\hat{f})$, then there is $z\in \R^2$ such that the pointwise rotation vector
\begin{equation}\rho(\hat{f},z)=\label{eq:lim} \lim_{n\to \infty} (\hat{f}^n(z)-z)/n \end{equation}
exists and coincides with $v$.

In this article, we consider area-preserving homeomorphisms, or more generally homeomorphisms preserving a Borel probability measure of full support. The question that we address is how the unbounded behavior of orbits along a given direction in the universal covering affects the rotation set. In principle, the rotation set does not distinguish fixed points from orbits that become unbounded at a sublinear rate; but does this kind of phenomenon exist?

%
%

An example given in \cite{kt-example} shows that the answer is positive in general: there is an ergodic area-preserving $C^\infty$ diffeomorphism with a lift whose rotation set is $\{(0,0)\}$, but such that almost every point has an unbounded orbit \emph{in every direction} (moreover, the orbit of almost every point visits every fundamental domain in the universal covering). However, in \cite{kt-pseudo} it is shown that this type of example is very particular, as it forces the existence of a large ``essential'' continuum of fixed points.
There are also examples where the rotation set consists of a single vector $w\notin \Q^2$, but there are orbits which are unbounded in the direction perpendicular to $w$ \cite{jager-bmm,kk-minimal}.

When the rotation set is not a single point, the situation turns out to be rather different. This is the case addressed in this article. The general idea behind our main result is illustrated by the theory of (area-preserving) twist maps. For twist maps of the cylinder $\T^1\times \R$, Mather showed in \cite{mather-twist} that either there exists a ``barrier'' (in this setting, an essential invariant circle) which impedes the drift of orbits in the vertical direction, or there is an orbit which converges towards infinity (moreover, the orbit converges toward the upper end in the future and the lower end in the past, or vice-versa).
For twist maps on $\T^2$, the later result was improved in \cite{sli-twist}: the non-existence of a ``barrier'' implies the existence of orbits with nonzero average speed of rotation in the direction perpendicular to the direction of twist (i.e. with nonzero \emph{shear} rotation number).

The result from \cite{sli-twist} was generalized in \cite{tal-dehn}, where the technical and restrictive twist condition was replaced  by the condition that the homeomorphism of $\T^2$ be in a Dehn homotopy class (which implies a weak type of topological twist). 

Our main result considers the case of area-preserving homeomorphisms of $\T^2$ homotopic to the identity. The twist condition is replaced by a weak condition on the existence of orbits with different average rotation speed (i.e. rotation vectors).

To be more precise, let us say that $\hat{f}$ has bounded displacement in the direction $v\in \R^2$ at a point $z\in \R^2$ if there exists $M>0$ such that
$$\abs{\smash{p_v(\hat{f}^n(z)-z)}}\leq M, \text{ for all } n\in \Z,$$
where $p_v(x) = \langle x;v/\norm{v}\rangle$ is the projection onto the direction $v$.
If the number $M$ can be chosen independent of $z\in \R^2$, then we say that $f$ has uniformly bounded displacement in the direction $v$. 

By an essential annulus we mean a subset $A$ of $\T^2$ homeomorphic to an open annulus and containing a homotopically nontrivial loop $\gamma$ of $\T^2$. The homological direction of $A$ is the direction of the line $\R v$, where $v\in \Z^2\simeq H_1(\T^2,\Z)$ represents the homology class of $\gamma$. In particular, a vertical annulus is one with homological direction $\{0\}\times \R$. The main result of this article is the following:

\begin{theoremain}\label{th:main} Let $f\colon \T^2\to \T^2$ be a homeomorphism homotopic to the identity preserving a Borel probability measure of full support, and suppose that some lift $\hat{f}\colon \R^2\to \R^2$ has a rotation set $\rho(\hat{f})=\{0\}\times [a,b]$, where $a<b$. Then $f$ has uniformly bounded displacement in the horizontal direction. Moreover, there exists an essential invariant annulus (which is necessarily vertical).
\end{theoremain}

A similar result was obtained by D\'avalos in \cite{davalos} under the assumption that there is a rational point in the rotation set which is not realized by a periodic orbit; however this hypothesis is disjoint from ours, since the existence of an invariant measure of full support already guarantees the realization of all rational points by periodic orbits (see Theorem \ref{th:reali}).
In \cite{tal-renato}, a vesion of Theorem \ref{th:main} was obtained under strong additional hypotheses (namely that the measure of full support be ergodic and with a non-rational rotation vector). 

When $f$ has some periodic essential annulus and $\rho(\hat{f})$ is not a single point, it is easy to verify that $\rho(\hat{f})$ is a segment of rational slope (parallel to the homological direction of the periodic annulus) containing some point of rational coordinates. As a consequence of Theorem \ref{th:main}, the converse of this fact holds:
 
\begin{theoremain}\label{th:main-gen} Let $f\colon \T^2\to \T^2$ be a homeomorphism homotopic to the identity preserving a Borel probability measure of full support, and $\hat{f}$ a lift of $f$ such that $\rho(\hat{f})$ has more than one point. Then $f$ has a periodic essential annulus if and only if $\rho(\hat{f})$ is a segment of rational slope containing some point of rational coordinates (and in this case the homological direction of the annulus is parallel to the rotation segment).
\end{theoremain}

Note that Theorem \ref{th:main} says that the dynamics of $f$ is essentially ``annular''; the horizontal homological direction of $\T^2$ plays no role in the dynamics. 
In particular, Theorem \ref{th:main} allows to translate existing results about rotation sets for homeomorphisms of  the annulus (\eg \cite{handel-annulus}) to this setting

%
%
%

A simple reformulation of Theorem \ref{th:main-gen} also leads to the following
\begin{corollarymain}\label{th:main-disp}  Let $f\colon \T^2\to \T^2$ be a homeomorphism homotopic to the identity, preserving a Borel probability measure of full support, and $\hat{f}\colon \R^2\to \R^2$ a lift of $f$. Suppose that $\rho(\hat{f})$ is not a single point, and for some nonzero $v\in \Z^2$,
$$\sup_{z\in \R^2, n\in \Z} \abs{\smash{p_v(\hat{f}^n(z)-z)}} = \infty$$
then $p_v(\rho(\hat{f})) \neq \{0\}$.
\end{corollarymain}
It seems reasonable to expect that, under the hypotheses of the previous theorem, $\sup p_v(\rho(\hat{f})) >0$ whenever $\sup_{z\in \R^2, n\in \Z} p_v(\hat{f}^n(z)-z)>0$. The main difficulty to conclude the latter fact is the following question: suppose that $\rho(\hat{f})$ has nonempty interior and is contained in $\{(x,y):x\leq 0\}$. If the boundary of $\rho(\hat{f})$ contains a vertical segment through the origin, does it follow that the horizontal displacement of orbits is uniformly bounded above? The arguments from the present article fail in this setting mainly due to the fact that the dynamics lifted to a horizontal cylinder (as done in \S\ref{sec:cylinder}) need not be nonwandering under these hypotheses. 

This article is organized as follows. Section 2 introduces some notation, terminology and preliminary results. In particular, it includes a consequence of \cite{KLN} about the possible rotation numbers of periodic points in the boundary of an invariant open topological disk, which is fundamental to obtain the uniformity in Theorem \ref{th:main}.
Section 3 contains the proof of the uniform boundedness of deviations in the horizontal direction under the hypotheses of Theorem \ref{th:main}. In Section 4, a general result about the realization of rational rotation vectors by periodic points in certain invariant continua is given. This is used in Section 5 to prove the existence of an invariant vertical annulus for a nonwandering homeomorphism whenever the horizontal displacement is uniformly bounded and the rotation set is not a point. Finally, in Section 6 the main theorems are proved using the preceding results.

\section{Notation and preliminaries}

The sets $\R_*$, $\R^2_*$, $\Z_*$ and $\Z^2_*$ denote the set of all non-zero elements of the corresponding spaces, \eg $\Z^2_* =\{v\in \Z^2:v\neq (0,0)\}$ and similarly for the other spaces. We denote by $P_1$ and $P_2$ the projections onto the first and second coordinates of $\R^2$, respectively.

Given $u,v\in \R^2$, we denote by $\langle u; v\rangle$ their Euclidean scalar product. The orthogonal projection onto the direction of $v$ is denoted by 
$$p_v(z) = \langle z ; v/\norm{v}\rangle.$$

From now on we assume that $f\colon \T^2\to \T^2$ is a homeomorphism homotopic to the identity.
\subsection{The rotation set}\label{sec:rotation}.

The rotation set, as defined in the introduction, satisfies the following properties (see \cite{m-z} and \cite[Lemma 4.2]{kk-reali})
\begin{enumerate} 
\item $\rho(\hat{f}^n+v) = n\rho(\hat{f})+v$ for any $n\in \Z$ and $v\in \Z^2$;
\item $\rho(A\hat{f}A^{-1})=A\rho(\hat{f})$ for any $A\in \GL(2,\Z)$.
\end{enumerate}
In particular, if $\rho(\hat{f})$ is a segment of rational slope, one may always find $A\in \GL(2,\Z)$ such that $\rho(A\hat{f}A^{-1})$ is a vertical segment (see \cite[Remark 2.5]{kk-reali}). Note that $A\hat{f}A^{-1}$ is a lift to $\R^2$ of a homeomorphism of $\T^2$ homotopic to the identity and conjugate to $\hat{f}$ (via the map $A_{\T^2}$ induced by $A$ on $\T^2$). 

\subsection{Essential and inessential sets}\label{sec:essential}
An open subset $U$ of $\T^2$ is said to be \emph{inessential} if every loop in $U$ is homotopically trivial in $\T^2$; otherwise, $U$ is \emph{essential}. An arbitrary set $E$ is called inessential if it has some inessential neighborhood. We say that $E$ is \emph{fully essential} if $\T^2\sm E$ is inessential.

%
%
%
%


\subsection{The sets $U_\epsilon'(z)$} \label{sec:uepsilon} 
Given $z\in \T^2$ and $\epsilon>0$, denote by $U_\epsilon'(z,f)$ (or simply $U_\epsilon'(z)$ when there is no ambiguity) the connected component of $\bigcup_{n\in \Z} f^n(B_\epsilon(z))$ containing $z$. Suppose that $f^n(B_\epsilon(z))$ intersects $B_\epsilon(z)$ for some $n\in \N$ (otherwise, $U_\epsilon'(z)=B_\epsilon(z)$). Since $f$ permutes the connected components of $\bigcup_{n\in \Z} f^n(B_\epsilon(z))$, it follows that $U_\epsilon'(z) = f^n(U_\epsilon'(z))$, and if $n\in \N$ is chosen minimal with that property, then $f^k(U_\epsilon'(z))$ is disjoint from $U_\epsilon'(z))$ whenever $1\leq k < n$. In particular, if $n>1$ then $U_\epsilon'(z)$ is disjoint from its image.





\subsection{Essential and inessential points} 
Following \cite{kt-ess}, we say that $x\in \T^2$ is inessential if there exists $\epsilon>0$ such that the set $U_\epsilon'(x)$ is inessential. Otherwise, we say that $x$ is essential. The set of inessential points is an open invariant set denoted by $\ine(f)$. Its complement is the set of essential points, denoted $\ess(f)$, which is a closed invariant set.


\subsection{Annular and strictly toral homeomorphisms}
Recall from \cite{kt-ess} that a nonwandering homeomorphism $f$ is called \emph{annular} if there is $v\in \Z^2_*$ and a lift $\hat{f}$ of $f$ such that $\hat{f}$ has uniformly bounded deviations in the direction $v$; and $f$ is called \emph{strictly toral} if $\fix(f^k)$ is not fully essential and $f^k$ is non-annular for each $k\in \N$.  The next result is contained in Theorem A and part (4) of Proposition 1.4 of \cite{kt-ess}.

\begin{theorem}[\cite{kt-ess}]\label{th:strictly} If $f$ is strictly toral, then $\ine(f)$ is inessential, and $\ess(f)$ is fully essential. Furthermore, if $x\in \ess(f)$ then $U_\epsilon'(x)$ is fully essential for any $\epsilon>0$.
\end{theorem}

We will also need the next result, which is Theorem D of \cite{kt-ess}.
\begin{theorem}[\cite{kt-ess}]\label{th:reali-toral}
 Let $f$ be a strictly toral homeomorphism with a lift $\hat{f}$, and suppose that for some $v\in \Z^2$ and $q\in \N$ there exists $z\in \R^2$ such that $\hat{f}^q(z)=z+v$. Then there exists $z'\in \R^2$ such that $\hat{f}^q(z')=z'+v$ and $\pi(z')\in \ess(f)$.
\end{theorem}

\subsection{Realization of rational rotation vectors}

The previous theorem is useful in combination with the following 

\begin{theorem}\label{th:reali} Suppose that $f$ preserves a Borel probability measure of full support on $\T^2$, and let $\hat{f}$ be a lift of $f$ to $\R^2$. If $\rho(\hat{f})$ is an interval and $v/q\in \rho(\hat{f})$ for some $v\in \Z^2$ and $q\in \N$, then there exists $z\in \R^2$ such that $\hat{f}^q(z)=z+v$.
\end{theorem}

This was proved by Franks for the case of area-preserving homeomorphisms \cite{franks-reali-area}. In \cite{kk-reali} the authors prove a version of the same theorem replacing the preservation of area by the curve intersection property (i.e. the property that any essential loop intersects its image by $f$). The latter version easily implies that Franks' theorem is still valid if one replaces the area-preserving hypothesis by the preservation of any Borel probability measure of full support, as stated here. Moreover, a more recent result of D\'avalos \cite{davalos} implies that the same result is true if $f$ is either nonwandering or non-annular.

\subsection{Nonwandering homeomorphisms}

We will use the following facts (for the first one, see \cite[Remark 4.1]{koro}).
\begin{proposition}\label{pro:nw-power} If $f$ is a nonwandering homeomorphism of a topological space, then so is $f^n$.
\end{proposition}

\begin{proposition}\label{pro:nw-lift} Let $f\colon S\to S$ be a nonwandering homeomorphism of a surface, and $\til{f}\colon \til{S}\to \til{S}$ a lift of $f$ by a finite covering $\til{\pi}\colon \til{S}\to S$. Then $\til{f}$ is nonwandering.
\end{proposition}
\begin{proof}
Let $U_1\subset \til{S}$ be a nonempty open set. Since $f$ is nonwandering, there exists $n_1\in \N$ and a deck transformation $T_1\in \deck(\til{\pi})$ such that $U_2:=U_1\cap T_1(\til{f}^{n_1}(U_1)) \neq \emptyset$. The fact that $f$ is nonwandering also implies that there is $n_2>n_1$ and $T_2\in \deck(\til{\pi})$ such that $U_2\cap T_2(\til{f}^{n_2}(U_2))\neq \emptyset$. Repeating this process we obtain recursively an increasing sequence $(n_i)_{i\in \N}$ of integers, a sequence $(T_i)_{i\in \N}$ of Deck transformations, and a decreasing sequence $(U_i)_{i\in \N}$ of nonempty open sets, such that 
$$U_{i+1} = U_i\cap T_i(\til{f}^{n_i}(U_i)).$$
Note that, for any $k\in \N$,
$$U_{k+1}\subset \bigcap_{i=1}^k T_i(\til{f}^{n_i}(U_i)) \subset \bigcap_{i=1}^k T_i(\til{f}^{n_i}(U_1)).$$
Let $k$ be larger than the number of elements of $\deck(\til{\pi})$. Then there exist $1\leq i<j\leq k$ such that $T_i=T_j$. This implies that $\til{f}^{n_i}(U_1)$ intersects $\til{f}^{n_j}(U_1)$, and so $\til{f}^{n_j-n_i}(U_1)\cap U_1\neq \emptyset$. This proves that $\til{f}$ is nonwnandering.
\end{proof}

\subsection{A result on prime ends rotation numbers}

We also need the following result, which is Theorem B of \cite{KLN} (stated in a simplified version). 
\begin{theorem}[\cite{KLN}] 
\label{th:kln}
Let $f\colon \R^2\to \R^2$ be an orientation preserving homeomorphism and $U\subsetneq \R^2$ be an open $f$-invariant topological disk. If $f$ is nonwandering in $U$ and the prime ends rotation number of $f$ in $U$ is not zero (mod $\Z$), then there are no fixed points of $f$ on the boundary of $U$.
Moreover, if $U$ is unbounded, then there are no fixed points of $f$ in the complement of $U$.
\end{theorem}
For the sake of brevity of the exposition, we will not enter into the details of prime ends, since we will use them only tangentially. We only mention that the prime ends compactification of the open topological disk $U\subset \R^2$ is obtained by attaching a circle called the circle of prime ends of $U$, thus obtaining a space $U\sqcup \T^1$ topologized in a way that it is homeomorphic to the closed unit disk.  If $f$ is a homeomorphism of $\R^2$ leaving $U$ invariant, then $f|_U$ extends to $U\sqcup \T^1$. The prime ends rotation number of $f$ in $U$, denoted $\rho(f,U)\in \R/\Z$, is the usual Poincar\'e rotation number of the orientation preserving homeomorphism induced on $\T^1$ by the extension of $f|_U$. 
For more details and definitions, we refer the reader to \cite{KLN}.

We will use a consequence of Theorem \ref{th:kln} for which we need a definition. 
Let $f$ be an orientation preserving homeomorphism of $\R^2$ leaving $U$ invariant and such that $f|_U$ is nonwandering. The latter condition implies that $f|_U$ has some fixed point $z_0$, due to a classic result of Brouwer \cite{fathi}. Consider the annulus $\R^2\sm \{z_0\}$, which we may identify with $\A=\T^1\times \R$ via a homeomorphism. We can regard $f$ as a homeomorphism of $\A\sqcup\{z_0\}\simeq \R^2$, where $z_0$ is the lower end of $\A$. Let $\pi_1\colon \R^2\to \A$ be the covering projection, and $\hat{f}\colon \R^2\to \A$ a lift of $f|_\A$. Since $f|_\A$ is isotopic to the identity, $\hat{f}$ commutes with the Deck transformations $(x,y)\mapsto (x+k,y)$, $k\in \Z$. If $z\neq z_0$ is a periodic point of $f$, we may define its rotation number associated to the lift $\hat{f}$ as follows: let $\hat{z}\in \pi_1^{-1}(z)$ be any point. Then there exist $k\in \Z$ and $m\in \N$ such that $\hat{f}^m(\hat{z})=\hat{z}+(k,0)$, and these numbers do not depend on the choice of $\hat{z}$. The rotation number $\rho(\hat{f}, z)$ is then defined to be $k/m$. Note that $\rho(\hat{f}+(k,0),z) = \rho(\hat{f}, z)+(k,0)$.

\begin{corollary}\label{coro:kln} All periodic points in $\bd U$ have the same rotation number associated to $\hat{f}$. If $U$ is unbounded, the same is true for all periodic points in $\R^2\sm U$.
\end{corollary}
\begin{proof}
Suppose for a contradiction that $f$ has two periodic points in $\bd U$ with different rotation numbers associated to a lift $\hat{f}$; \ie there are $\hat{x}_0, \hat{x}_1\in \R^2$ such that $\pi_1(\hat{x}_i)\in \bd U$ and $\hat{f}^{m_i}(\hat{x}_i) = \hat{x}_i+(k_i,0)$, with $k_0/m_0\neq k_1/m_1$. Using $f^{m_0m_1}$ instead of $f$ (which is also nonwandering in $U$ by Proposition \ref{pro:nw-power}), we may assume that $m_1=m_2=1$. Moreover, we may assume that $k_0=0$ by choosing the lift $\hat{f}$ appropriately. Thus, $\hat{x}_0$ is a fixed point of $\hat{f}$ and $\hat{f}(\hat{x}_1) = \hat{x}_1+(k_1,0)$ with $k_1\neq 0$.

Let $\til{\A} = (\R/2k_1\Z)\times \R$, and $\tau\colon \R^2\to \til{\A}$ the covering projection. Note that $\pi_1\colon \R^2\to \A$ induces on $\til{\A}$ a finite covering $\til{\pi}_1\colon \til{\A}\to \A$, whose group of Deck transformations is generated by the map $(x,y)\mapsto (x+1+\R/2k_1\Z, y)$. In particular, the homeomorphism $\til{T}(x,y) = (x+k_1+\R/2k_1\Z, y)$ on $\til{\A}$ is a Deck transformation of $\til{\pi}_1$ such that $\til{T}^2=\id$.

The map $\hat{f}$ induces on $\til{\A}$ a homeomorphism $\til{f}$, which is a lift of $f|_\A$ by $\til{\pi}_1$. Since $f|_\A$ is homotopic to the identity, $\til{f}$ commutes with the covering transformations and in particular with $\til{T}$. Denote by $\til{\A}_* = \til{\A}\sqcup\{z_0^*\}$ where $z_0^*$ is the lower end of $\til{\A}$, topologized in the usual way so that $\til{\A}_*\simeq \R^2$. We can extend $\til{f}$ and $\til{T}$ and regard them as maps of $\til{\A}_*$ by fixing $z_0^*$. 

Let $\til{U} = \til{\pi}^{-1}(U\sm z_0)\sqcup\{z_0^*\}$. Then $\til{U}$ is an open $\til{f}$-invariant and $\til{T}$-invariant topological disk, and $\til{U}$ is unbounded in $\til{\A}_*$ if and only if $U$ is unbounded in $\R^2$.
The fact that $f|_{U\sm \{z_0\}}$ is nonwandering implies that $\til{f}|_{\til{U}\sm\{z_0^*\}}$ is also nonwandering due to Proposition \ref{pro:nw-lift}. 

Since $\til{T}^2$ is the identity and $\til{T}$ has no fixed points, one easily sees that the map induced on the prime ends of $\til{U}$ by $\til{T}$ has a rotation number $1/2$ (mod $\Z$) (for instance, any accessible prime end of $\til{U}$ is periodic of period $2$ and non-fixed by the map induced by $\til{T}^2$ on the circle of prime ends).

Since $\til{f}$ commutes with $\til{T}$, the respective maps induced on the prime ends of $\til{U}$ also commute. It is an easy consequence of the definition of rotation number that, if two orientation-preserving circle homeomorphisms commute, then the rotation number of their composition is the sum of their rotation numbers (mod $\Z$). 
Thus, $$\rho(\til{T}\til{f},\til{U}) = \rho(\til{f},\til{U})+ 1/2 \text{ (mod $\Z$)}.$$

In particular, if $\rho(\til{f},\til{U})=0$, then $\rho(\til{T}\til{f}, \til{U})=1/2 \neq 0$ (mod $\Z$). Note that $\til{T}\til{f}$ is another lift of $f|_\A$ to $\til{\A}$, and therefore is also nonwandering in $U$. Let $\til{x}_0=\tau(\hat{x}_0)$ and $\til{x}_1=\tau(\hat{x}_1)$. Then $\til{x}_0$ is a fixed point of $\til{f}$, and $\til{x}_1$ is a fixed point of $\til{T}\til{f} = \til{T}^{-1}\til{f}$. Moreover, both $\til{x}_0$ and $\til{x}_1$ belong to $\bd \til{U}$. Thus both maps $\til{f}$ and $\til{T}\til{f}$ have a fixed point in $\bd \til{U}$, but one of them has a nonzero prime ends rotation number, contradicting Theorem \ref{th:kln}.

This proves the first claim of the theorem. In the case that $U$ is unbounded, the same proof by contradiction applies, choosing $\hat{x}_i$ in $\R^2\sm U$ instead of $\bd U$.
\end{proof}

\section{Proof of Theorem \ref{th:main}: uniformly bounded deviations}
We will divide the proof of Theorem \ref{th:main} in two parts. This section is devoted to the proof of the first part, which is contained in the following proposition. The proof of the existence of an essential invariant annulus is postponed to Section \ref{sec:th:annulus}.

\begin{proposition}\label{pro:main-A} Under the hypotheses of Theorem \ref{th:main}, there exists $M>0$ such that
$$\abs{\smash{P_1(\hat{f}^n(z) - z)}}\leq M \text{ for all } z\in \R^2,\, n\in \Z.$$
\end{proposition}
The remainder of this section contains the proof of this proposition.

Assume that $f$ is under the hypotheses of Theorem \ref{th:main}, \ie $f$ preserves a Borel probability measure $\mu$ of full support and $\rho(\hat{f})$ is a vertical nondegenerate interval containing the origin. 

\begin{claim} We may assume that $\{0\}\times [-1,1]\subset \rho(\hat{f})$.
\end{claim}
\begin{proof}
Note that $\rho(\hat{f}^k-v) = k\rho(f)-v$ for any $k\in \Z$ and $v\in \Z^2$ (see \cite{m-z}), and so letting $\hat{g} = \hat{f}^k-(0,l)$ for appropriately chosen integers $k,l$ we have that $\{0\}\times [-1,1]\subset \rho(\hat{g})$, and $\hat{g}$ is a lift of $f^k$. It is easy to verify that $\hat{g}$ has uniformly bounded deviations in the horizontal direction if and only if so does $\hat{f}$. Thus we may use $f^k$ and $\hat{g}$ in place of $f$ and $\hat{f}$.
\end{proof}

\begin{claim} There exist points $z_{-1}$, $z_0$, $z_{1}$ in $\R^2$ such that $\hat{f}(z_i) = z_i + (0,i)$ for $i\in \{-1,0,1\}$
\end{claim}
\begin{proof} It follows from Theorem \ref{th:reali} and from our previous assumption.
\end{proof}

The proof of Proposition \ref{pro:main-A} will be by contradiction. Thus, from now on we assume that $\hat{f}$ has unbounded deviations in the horizontal direction, \ie
$$\sup\{ \abs{\smash{P_1(\hat{f}^n(z) - z)}} : z\in \R^2,\, n\in \Z\} = \infty.$$
Under this assumption, we have 
\begin{claim} $f$ is strictly toral.
\end{claim}
\begin{proof}
Let us first show that $f^k$ is non-annular for any $k\in \N$. By Theorem \ref{th:reali} and the fact that $\rho(\hat{f})$ contains the origin, we know that $\hat{f}$ has a fixed point. Thus, if $f^k$ is annular for some $k\in \N$, then so is $f$ (by \cite[Proposition 1.4(5)]{kt-ess}). The fact that $\rho(\hat{f})$ is a nondegenerate vertical interval easily implies that $f$ can only be annular if $\hat{f}$ has uniformly bounded deviations in the horizontal direction, which by our assumptions is not possible. 

Now suppose that $\fix(f^k)$ is fully essential for some $k\in \N$. Since $f^k$ is non-annular, by \cite[Proposition 5.1]{kt-ess} applied to $f^k$, some lift (and thus any lift) of $f^k$ to $\R^2$ has a rotation set consisting of a single point. This is a contradiction, since $\rho(\hat{f}^k) = k\rho(\hat{f})$ is a nondegenerate interval. 

Thus $\fix(f^k)$ is not fully essential, and $f^k$ is non-annular for any $k\in \N$, which means that $f$ is strictly toral.
\end{proof}

\begin{claim} We may assume that $\pi(z_i)\in \ess(f)$ for each $i\in \{-1,0,1\}$. 
\end{claim}
\begin{proof} It follows from Theorem \ref{th:reali-toral}.
\end{proof}

\subsection{The sets $\omega_+$ and $\omega_-$}

The definitions and properties that we introduce here were already used in \cite{tal-trans, kt-pseudo}. Since we work with the vertical direction as a reference, we use a simplified notation.

Denote by $H^+$ and $H^-$ the half-planes $\{(x,y) : x\geq 0\}$ and $\{(x,y):x\leq 0\}$, respectively. Let $f\colon \T^2\to \T^2$ be a homeomorphism homotopic to the identity, and $\hat{f}\colon \R^2\to \R^2$ a lift of $f$. Define the set $\omega_+\subset \R^2$ as the union of all unbounded connected components of
$$\bigcap_{i=-\infty}^{\infty}\hat f^{i}(H^{+}).$$
The set $\omega_-$ is defined analogously using $H^-$ instead of $H^+$.

The following properties are easy consequences of the definitions. 
\begin{proposition}\label{pro:omegapro} The sets $\omega_\pm$ are closed, and
\begin{itemize}
\item $\hat{f}(\omega_\pm)=\omega_\pm$;
\item $\omega_\pm + (0,b) \subset \omega_\pm$ for all $b\in \Z$;
\item $\omega_+ + (a,0) \subset \omega_+$ and $\omega_- - (a,0)\subset \omega_-$ for all $a\in \N$;
\item $\omega_\pm$ is non-separating, and its complement is simply connected.
\item Every connected component of $\omega_+$ is unbounded to the right, and every connected component of $\omega_-$ is unbounded to the left.
\end{itemize}
\end{proposition}

\begin{claim} $\omega_+$ and $\omega_{-}$ are non-empty.
\end{claim}
\begin{proof} Since the origin lies in the boundary of the rotation set, this is a direct consequence of Lemma 4 and Corollary 1 of \cite{tal-trans}.
\end{proof}

The following property holds whenever $f$ is strictly toral and the sets $\omega_-$ and $\omega_+$ are nonempty:
\begin{proposition} $\ess(f) \subset \ol{\pi(\omega_+)}\cap \ol{\pi(\omega_-)}$
\end{proposition}
\begin{proof}
If $z \in \ess(f)$, then by Theorem \ref{th:strictly} we know that $U_\epsilon'(z)$ is fully essential for any $\epsilon>0$. This implies that every connected component of $\pi^{-1}(\T^2\sm U_\epsilon'(z))= \R^2\sm \pi^{-1}(U_\epsilon'(z))$ is bounded \cite[Proposition 1.3]{kt-ess}. Since $\omega_+$ is nonempty and each connected component of $\omega_+$ is unbounded, it follows that $\omega_+\cap \pi^{-1}(U_\epsilon'(z))\neq \emptyset$. In other words, $\pi(\omega_+)$ intersects $U_\epsilon'(z)$. Since $\pi(\omega_+)$ is invariant, the definition of $U_\epsilon'(z)$ implies that $\pi(\omega_+)$ intersects $B_\epsilon(z)$. Since this holds for any $\epsilon>0$, we conclude that $z\in \ol{\pi(\omega_+)}$. The same argument applied to $\omega_-$ in place of $\omega_+$ completes the proof.
\end{proof}

Our assumption about the rotation set allows us to prove more:
\begin{claim}\label{claim:omegasetocam} $\pi(\omega_+)\cap \pi(\omega_-)\neq \emptyset$.
\end{claim}
\begin{proof}
Suppose for a contradiction that $\pi(\omega_+)\cap \pi(\omega_-)=\emptyset$.
Note that by our assumptions, $\hat{f}(z_1)=z_1+(0,1)$ and $\pi(z_1)\in \ess(f)$. Fix $\epsilon<1/4$ such that $B=B_\epsilon(z_1)$ satisfies $\hat{f}(B)\cap B=\emptyset$. By the previous claim, $\pi(B)$ intersects $\pi(\omega_+)$ and $\pi(\omega_-)$, which means that there exist $p_1, p_2\in \Z$ such that $\omega_-+(p_1,0)$ and $\omega_++(p_2,0)$ both intersect $B$. If $\ell_0\subset B$ is a straight line segment (not including its endpoints) joining a point of $\omega_-+(p_1,0)$ to a point of $\omega_++(p_2,0)$, then one can find a sub-segment $\ell\subset \ell_0$ that has the same property as $\ell_0$ but in addition is disjoint from $\omega_-+(p_1,0)$ and $\omega_++(p_2,0)$. 

This means that the hypotheses of \cite[Proposition 4.10]{kt-pseudo} hold. In particular, part (5) of said proposition implies that one of the following options holds:
\begin{itemize}
\item[(1)] For each $x\notin (\omega_-+(p_1,0)) \cup (\omega_++(p_2,0))$, there is $m_0\in \Z$ and $\delta>0$ such that $f^k(x)$ is disjoint from $B_\delta(x+(0,m))$ whenever $m<m_0$ and $k> 0$, or
\item[(2)] A similar property as (1), with $k<0$ instead  of $k>0$.
\end{itemize}
Suppose first that (1) holds. Note that since $\pi(\omega_+)\cap \pi(\omega_-)=\emptyset$, the point $\pi(z_{-1})$ is disjoint from either $\pi(\omega_+)$ or $\pi(\omega_-)$. Suppose $\pi(z_{-1})\notin \pi(\omega_+)$ (the other case is analogous). Choose $z_{-1}'\in \pi^{-1}(\pi(z_{-1}))$ (\ie an integer translate of $z_{-1}$) such that the first coordinate of $z_{-1}'$ is greater than $p_1$. This choice guarantees that $z_{-1}'\notin \omega_-+(p_1,0)$, and since $\pi(z_{-1}')\notin \pi(\omega_+)$ we also have that $z_{-1}'\notin \omega_++(p_2,0)$. Thus, setting $x=z_{-1}'$ in (1) we conclude that $\hat{f}^k(x)$ is disjoint from $B_\delta(x+(0,m))$ whenever $m\leq m_0$ and $k>0$. This contradicts the fact that $\hat{f}^k(x) = x-(0,k)$ for all $k\in \N$ (since $x$ is an integer translate of $z_{-1}$).

If (2) holds instead, a similar argument using $z_1$ instead of $z_{-1}$ leads to a contradiction. 
Since all cases lead to a contradiction, the proof is complete.
\end{proof}

\subsection{Chains of disks with a special property}
Note that by Claim \ref{claim:omegasetocam}, there exists $m_0$ such that $(\omega_+-(m_0,0))\cap \omega_-\neq \emptyset$, and so a similar property holds if one replaces $m_0$ by any $m\geq m_0$. Let 
$$F_m:=(\omega_+ - (m,0))\cup \omega_{-},$$
and given $\hat{z}\in \R^2\sm F_m$ denote by $O_m(\hat{z})$ the connected component of $\R^2\sm F_m$ containing $\hat{z}$. Recall from \cite[Proposition 4.11]{kt-pseudo} the following facts:
\begin{itemize}
\item  $O_m(\hat{z})$ is an open topological disk for any $m\in \N$, and 
\item  $O_m(\hat{z})\cap (O_m(\hat{z})+(0,k))=\emptyset$ for any integers $k\neq 0$ and $m\geq m_0$.
\end{itemize}
Note that $(F_m)_{m\in \N}$ is an increasing chain of sets, so $(O_m(\hat{z}))_{m\in \N}$ is a decreasing chain of disks. 

\begin{lemma}\label{lem:chain} For any $z\in \T^2\sm \pi(\omega_+)$ there exist $\ol{w}\in \Z^2_*$ and $\hat{z}\in \pi^{-1}(z)$ such that 
\begin{itemize}
\item[(1)] $\hat{z}\notin F_m$ for any $m\in \N$, and
\item[(2)] for any $v\in \Z^2_*$ not parallel to $\ol{w}$, there is $m\geq m_0$ such that $$O_m(\hat{z})\cap (O_m(\hat{z})+v)=\emptyset.$$
\end{itemize}
\end{lemma}

\begin{proof}

Choose $\hat{z}\in \pi^{-1}(z)$ such that $P_1(\hat{z})>0$. Since $P_1(\omega_-)\subset (-\infty,0]$, it follows that $\hat{z}\notin \omega_-$. Moreover, $\hat{z}\notin \omega_+-(m,0)$ for any $m\in \Z$ since $z=\pi(\hat{z})\notin \pi(\omega^+)$ by hypothesis. In particular, $\hat{z}\notin F_m$ for any $m\in \N$.

For $m\geq m_0$, let $O_m=O_m(\hat{z})$. As we already mentioned, by \cite[Proposition 4.11]{kt-pseudo} we have that $O_m$ is an open topological disk and $O_m\cap (O_m+(0,k))=\emptyset$ for all integers $k\neq 0$. Suppose for a contradiction that (2) does not hold. Then there exist two non-parallel vectors $u, v$ in $\Z^2_*$ such that $O_m\cap (O_m+u)\neq \emptyset$ and $O_m\cap (O_m+v)\neq \emptyset$ for all $m\in \N$. If $u=(u_1,u_2)$ and $v=(v_1,v_2)$, then the previous remarks imply that $u_1\neq 0$ and $v_1\neq 0$. 
We may assume without loss of generality that $u_1>0$ and $v_1>0$ (otherwise we replace $u$ by $-u$ or $v$ by $-v$ as necessary). 

Let $$W_m = \bigcup_{i=0}^{v_1} O_m+iu\cup \bigcup_{j=0}^{u_1} O_m+jv.$$
Note that, by our assumption, $W_m$ is connected. Moreover, $W_m$ contains $\hat{z}+v_1u$ and $\hat{z}+u_1v$, and so $W_m$ intersects $W_m+(v_1u-u_1v) = W_m+(0,k)$, where $k=v_1u_2-u_1v_2\in \Z$. The fact that $u$ and $v$ are not parallel implies that $k\neq 0$.

On the other hand, since $O_m$ is disjoint from $F_m$, we have that for $i\in \N$ the set $O_m+iu$ is disjoint from $$F_m+iu = F_m+(iu_1, 0) = \big(\omega_-+(iu_1,0)\big)\cup \big(\omega_++(iu_1-m,0)\big),$$
where we used the fact that $F_m+(0,n)=F_m$ for any $n\in \Z$. Moreover, since $\omega_-\subset \omega_- +(iu_1,0)$ (because $iu_1>0$), we conclude that $O_m+iu$ is also disjoint from $F_{m-iu_1} = \omega_-\cup(\omega_++(iu_1-m,0))$.

In particular, if $0\leq i\leq v_1$, we have that $\omega_+ + (v_1u_1-m,0)\subset \omega_++(iu_1-m,0)$, so that $F_{m-v_1u_1}\subset F_{m-iu_1}$. Hence $O_m+iu$ is also disjoint from $F_{m-v_1u_1}$. A similar argument shows that $O_m+jv$ is disjoint from $F_{m-v_1u_1}$ if $0\leq j\leq u_1$. 

Thus we conclude that $W_m\subset \R^2\sm F_{m-v_1u_1}$. In particular, using $m = m_0+v_1u_1$ we see that $W_m\subset \R^2\sm F_{m_0}$.  Since $W_{m}$ is connected and contains $\hat{z}$, the definition of $O_{m_0}$ implies that $W_m\subset O_{m_0}$. But we have shown that $W_m$ intersects $W_m+(0,k)$ where $k=v_1u_2-u_1v_2\neq 0$, while $O_{m_0}$ is disjoint from $O_{m_0}+(0,k)$. This contradiction proves our claim.
\end{proof}

Using the terminology of \cite[Sec. 3]{kt-pseudo}, the Lemma \ref{lem:chain} says that $(O_m(\hat{z}))_{m\in \N}$ is an eventually $\Z^2\sm \R \ol{w}$-free chain of open topological disks. 

\subsection{Localizing the points $z_i$ in $\omega_+$}
Let us state a general fact \cite[Proposition 3.3]{kt-pseudo}. The only hypothesis used in this proposition is that $\hat{f}$ is a lift of a homeomorphism homotopic to the identity.
\begin{proposition}\label{pro:cadeia-anular} Suppose that for some $\ol{w}\in \Z^2_*$ there exists an eventually $\Z^2\sm \R\ol{w}$-free decreasing chain of open connected sets $(O_m)_{m\in \N}$ such that $\hat{f}(O_m)=O_m$ and $\pi(O_m)$ is fully essential for each $m\in \N$. Then $\hat{f}$ has uniformly bounded deviations in some direction, \ie there exists $w\in \R^2_*$ and $M>0$ such that 
$$\abs{\smash{p_w(\hat{f}^n(z)-z)}} \leq M\quad \text{for all } n\in \Z,\, z\in \R^2.$$
\end{proposition}

In our current setting, this allows us to obtain the following:

\begin{claim} $\pi(z_i)\in \pi(\omega_+)$ for each $i\in \{-1,0,1\}$.
\end{claim}
\begin{proof}
Fix $i\in \{-1,0,1\}$. We will assume that $z_i$ is fixed for $\hat{f}$. Otherwise, we may replace $\hat{f}$ by the new lift $\hat{f}-(0,i)$ of $f$, which fixes $z_i$, and proceed with the same proof (note that this leaves the sets $\omega_\pm$ unchanged).

Assume for contradiction that $\pi(z_i)\notin \pi(\omega_+)$. Then Lemma \ref{lem:chain} implies that there is an integer translate $z_i'$ of $z_i$ and some $\ol{w}\in \Z^2_*$ such that $z_i'\notin F_m$ for all $m\in \N$ and $(O_m(z_i'))_{m\in \N}$ is an eventually $\Z^2\sm \ol{w}\Z$-free chain of disks. Moreover, as we already saw, $\pi(z_i')\in \ess(f)$, so that $\pi(O_m(z_i'))$ is fully essential, for each $m\in \N$. In addition, $O_m(z_i')$ is invariant, since it is a connected component of the complement of an invariant set containing $z_i'$, which is fixed by $\hat{f}$. Thus the hypotheses of Proposition \ref{pro:cadeia-anular} hold, and we conclude that $\hat{f}$ has uniformly bounded deviations in some direction. This is not possible due to the fact that $\rho(\hat{f})$ is a nondegenerate vertical interval (which implies unbounded deviations in all non-horizontal directions) and by our assumption the deviations in the horizontal direction are not uniformly bounded. This contradiction proves the claim.

%
%
\end{proof}

\subsection{The dynamics on the cylinder}\label{sec:cylinder}
Denote by $T_1, T_2$ the translations $(x,y)\mapsto (x+1,y)$ and $(x,y)\mapsto (x,y+1)$, respectively.
We will consider the cylinder $\til{\A} = \R^2/\langle T_2\rangle \simeq \R\times \T^1$. Let $\tau\colon \R^2\to \til{\A}$ be the covering projection. The covering $\pi\colon \R^2 \to \T^2$ also induces a covering $\til{\pi}\colon \til{\A}\to \T^2$ naturally, and the lift $\hat{f}$ induces a lift $\til{f}$ of $f$ such that the following diagram commutes. 
\SelectTips{xy}{12}
\[
\xymatrix{ 
 & \R^2\ar^{\hat{f}}[rr]\ar_\tau[ld]\ar'[d][dd] && \R^2\ar_\tau[ld]\ar^\pi[dd] \\
\til{\A}\ar^{\quad\quad\til{f}}[rr]\ar_{\til{\pi}}[rd]  && \til{\A}\ar^{\til{\pi}}[rd]\\
& \T^2\ar^f[rr] && \T^2}
\]

Consider the set $\til{\omega}_+ = \tau(\omega_+)$. Then $\til{\omega}_+$ is $\til{f}$-invariant, and every connected component of $\til{\omega}_+$ is unbounded to the right (and bounded to the left) in $\til{\A}$. Moreover, the fact that $\omega_+$ is closed and $T_2$-invariant implies that $\til{\omega}_+$ is closed in $\til{\A}$.

Let $\til{\A}_* = \til{\A}\sqcup \{L_\infty\}$ the space obtained by adding to $\til{\A}$ one of its topological ends (the one on the left, $L_\infty$) and topologized so that $\til{\A}_*\simeq \R^2$. This is formally done by letting $L_\infty$ represent an arbitrary new point, and defining an open set of $\til{\A}_*$ to be any union of open sets of $\til{\A}$ with sets of the form $\{(x,y)\in \til{\A}:x<a\}\sqcup\{L_\infty\}$ for some $a\in \R$.  

Being homotopic to the identity, the map $\til{f}$ leaves the two ends of $\til{\A}$ invariant, and in particular $\til{f}$ extends to a map $\til{f}_*$ of $\til{\A}_*$ by fixing the point $L_\infty$. 
Let $U_*$ be the connected component of $\til{\A}_*\sm \til{\omega_+}$ which contains $L_\infty$. Then $U_*$ is an open $\til{f}_*$-invariant set, and $U_*$ is also simply connected (since $\til{\A}_*\simeq \R^2$ and every connected component of the complement of $U_*$ is unbounded). 

Moreover, $U_*$ itself is unbounded in $\til{\A}_*$. Indeed, in the case that $U_*$ is bounded, it follows  that $U=U_*\sm \{L_\infty\}$ is bounded to the right in $\til{\A}$, and so $\tau^{-1}(U)$ is bounded to the right in $\R^2$. Thus, if $H^-$ denotes the half-plane $\{(x,y):x\leq 0\}$ in $\R^2$, then $\hat{U}=\tau^{-1}(U)$ is an open $\hat{f}$-invariant set such that $\hat{U}\subset H^- + (m,0)$ for some $m\in \N$. Furthermore, $H^--(1,0)\subset \hat{U}$, due to the definition of $\omega_+$ and $U_*$. Thus $H^--(1,0)\subset \hat{U}\subset H^-+(m,0)$, and this easily implies that $\hat{f}$ has uniformly bounded displacement in the horizontal direction (see for instance \cite[Proposition 1.5]{kt-ess}), which is not possible under our assumptions.

Thus $\til{f}_*$ is a homeomorphism of a plane, preserving orientation (since $\til{f}$ is homotopic to the identity) and leaving and unbounded open topological disk $U_*$ invariant.

Note that in the previous section we showed that $\pi(z_i)\in \pi(\omega_+)$, so in particular replacing $z_i$ by some integer translate we may assume that $z_i\in \omega_+$ for each $i\in \{-1,0,1\}$. This means that if $\til{z}_i=\pi(z_i)$, then $\til{z}_i\in \til{\A}_*\sm U_*$. Recall that $\hat{f}(z_i)=z_i+(0,i)$, so that the rotation number of the fixed point $\til{z}_i$ of $\til{f}$ relative to the lift $\hat{f}$ is $i$ (see the comments before Corollary \ref{coro:kln}, observing that the cylinder $\A$ used there is vertical instead of horizontal).

We would like to apply Corollary \ref{th:kln} to $\til{f}_*$ and $U_*$. To meet the hypotheses of the corollary, we first need to verify that $\til{f}_*$ is nonwandering (at least in $U_*$). 

\begin{claim}\label{claim:nw} $\til{f}$ is nonwandering in $\til{\A}$.
\end{claim}
\begin{proof}
Recall that we are assuming that $f$ preserves a Borel probability measure of full support $\mu$ on $\T^2$. There is a corresponding lifted (non-finite) $\hat{f}$-invariant measure $\hat{\mu}$ on $\R^2$ which is invariant by $\Z^2$-translations, positive on open sets, and finite on compact sets ($\hat{\mu}$ satisfies $\hat{\mu}(E) = \mu(\pi(E))$ for any Borel set $E$ such that $\pi|_E$ is injective). This induces a measure $\til{\mu}$ on $\til{\A}$ which is also $\til{f}$-invariant, finite on compact sets, positive on open sets and $\til{T}_1$-invariant. 

Suppose there is an open set $V\subset \til{\A}$ which is wandering for $\til{f}$. By choosing a relatively compact open subset of $V$, we may assume that $\til{\mu}(V)<\infty$. Moreover, we may assume that $V\subset Q$, where $Q=\{(x,y)\in \til{\A}:0\leq x<1\}$, by using a suitable $\til{T}_1$-translation of $V$ intersected with $Q$.
Letting $C_k = \{(x,y)\in \til{\A} : -k\leq x< k\}$, we have that $\til{\mu}(C_k) = 2k\til{\mu}(Q)$. On the other hand, the $\til{\mu}$-measure of $V_n = \bigcup_{i=1}^n \til{f}^i(V)$ is $n\til{\mu}(V)$. 

Let $r = \til{\mu}(V)/(2\til{\mu}(Q))$, and let $k_n$ be the integer such that $nr-1\leq k_n< nr$. Then $\til{\mu}(C_{k_n}) < \mu(V_n)$, and so there is a point of $V_n$ in the complement of $C_{k_n}$. This means that there exists $z_n\in V$ and an integer $1\leq i_n\leq n$ such that $\til{f}^{i_n}(z_n)\notin C_{k_n}$, which implies that $$\abs{\smash{\til{P}_1(\til{f}^{i_n}(z_n))-\til{P}_1(z_n)}} \geq k_n-1,$$ where $\til{P}_1\colon \til{\A}\to \R$ is the projection onto the first coordinate.

Letting $\hat{z}_n$ be an element of $\til{\pi}^{-1}(z_n)$ in $[0,1]^2$, we see that $$\abs{P_1\bigg(\frac{\hat{f}^{i_n}(\til{z}_n)-\til{z}_n}{i_n}\bigg)} \geq \frac{k_n-1}{i_n} \geq \frac{nr-2}{n} \xrightarrow[n\to \infty]{} r.$$
Choosing subsequences, we may assume that $(\hat{f}^{i_n}(\til{z}_n)-\til{z}_n)/i_n$ converges to some $v\in \R^2$, which by definition must be in $\rho(\hat{f})$. Our previous observation implies that the first coordinate of $v$ has modulus at least $r$. In particular, the first coordinate of $v$ is nonzero, contradicting the fact that $\rho(\hat{f})$ is a vertical interval containing the origin.
\end{proof}

The previous claim implies that $\til{f}_*$ is nonwandering in $U_*$, so we are under the hypotheses of Corollary \ref{coro:kln}, using $\til{f}_*$ as the map $f$, and $L_\infty$ as $z_0$. Since we already saw that $U_*$ is unbounded, we conclude that all periodic points of $\til{f}_*$ in $\til{A}_*\sm U_*$ have the same rotation number relative to the lift $\hat{f}$. This contradicts the fact that $\til{z}_0$ and $\til{z}_1$ are fixed points of $\til{f}$ in $\til{\omega}_+\subset \til{A}_*\sm U_*$ with respective rotation numbers $0$ and $1$. This contradiction completes the proof of Proposition \ref{pro:main-A}.
\qed
\begin{remark} The only point of the proof where the existence of an invariant measure of full support is essentially used is in the proof of Claim \ref{claim:nw}. For all other proofs in this article it suffices to assume that $f$ is nonwandering.
\end{remark}

\section{A realization theorem for fully essential invariant continua}
This result improves Theorem D of \cite{kt-ess}.

\begin{theorem}\label{th:reali-full} Let $f$ be a nonwandering homeomorphism homotopic to the identity, $\hat{f}$ a lift of $f$ to $\R^2$, and $K\subset \T^2$ a fully essential $f$-invariant continuum. Given $v\in \Z^2$ and $q\in \N$, if there exists $z\in \R^2$ such that $\hat{f}^q(z')=z'+v$, then there also exists $z'\in \R^2$ such that $\pi(z')\in K$ and $\hat{f}^q(z')=z'+v$. 
\end{theorem}

%

\begin{proof}
\setcounter{claim}{0}
Let $\hat{g}=\hat{f}^q-v$, so that $\hat{g}$ is a lift of $g = f^q$ and by hypothesis $\hat{g}$ has at least one fixed point. We would like to show that $\hat{g}$ has a fixed point $z'$ such that $\pi(z')\in K$.

Fix $m\in \N$ so large that the only element of $m\Z^2\cap \rho(\hat{g}) = \{(0,0)\}$ (the origin must belong to $\rho(\hat{g})$ since $\hat{g}$ has a fixed point). Consider the map $\til{g}$ induced by $\hat{g}$ on the torus $\til{\T}^2=\R^2/(m\Z^2)$. Denote by $\tau\colon \R^2\to \R^2/(m\Z^2)$ the quotient projection. There is a finite covering $\til{\pi}\colon \til{\T}^2\to \T^2$ defined by the projection $\til{\T}^2\to \til{\T}^2/(\Z^2/m\Z^2) \simeq \T^2$, and $\til{g}$ is a lift of $g$ by $\til{\pi}$. 

Note that $z\in \R^2$ is a fixed point of $\hat{g}$ if and only if $\tau(z)$ is a fixed point of $\til{g}$. Thus, to prove the theorem it suffices to show that $\til{g}$ has a fixed point in the compact $\til{g}$-invariant set $\til{K}=\til{\pi}^{-1}(K)$. 

Observe that $\til{K}$ is a fully essential continuum in $\til{\T}^2$. Indeed, suppose that some connected component $U$ of $\til{\T}^2\sm \til{K}$ is essential in $\til{\T}^2$. Since the map induced by $\til{\pi}$ in the fundamental groups is injective (explicitly, it is the inclusion $m\Z^2\to \Z^2$), it follows that $\til{\pi}(U)$ is connected and essential in $\T^2$. But since $\til{\pi}(U)$ is disjoint from $K$, this contradicts the fact that $K$ is fully essential. 

We need to show that $\til{K}$ contains a fixed point of $\til{g}$. Suppose for a contradiction that $\til{K}\cap \fix(\til{g})=\emptyset$. 
Then the connected components of $\til{\T}^2\sm \til{K}$ (which are open topological disks) form an open covering of the compact set $\fix(\til{g})$, and so we can find finitely many such components $U_1,\dots, U_k$ such that $\fix(\til{g})\subset U_1\cup\cdots\cup U_k$. Moreover, we may chose them so that each $U_i$ intersects $\fix(\til{g})$,  and in particular each $U_i$ is $\til{g}$-invariant.


Note that $\til{g}$ is nonwandering due to Proposition \ref{pro:nw-lift}, and $\bd U_i\subset \til{K}$ has no fixed points of $\til{g}$. These two facts and a classic result of Cartwright and Littlewood allow to conclude that the extension of $\til{g}|_{U_i}$ to the prime ends compactification of $U_i$ has no fixed points in the circle of prime ends, and therefore there exists a closed topological disk $D_i\subset U_i$ such that $\fix(\til{g})\cap U_i\subset D_i$ and the fixed point index of $g$ in $D_i$ is $1$ (this argument is contained in Proposition 4.2 of \cite{koro}). Thus, the Lefschetz index of $\til{g}$ is exactly $k$. Since by our hypothesis $\hat{g}$ (and so $\til{g}$) has at least one fixed point, $k\geq 1$. But on the other hand the Lefschetz index of a homeomorphism of $\T^2$ homotopic to the identity must be $0$, so we have a contradiction.
\end{proof}

\section{Existence of an essential invariant annulus}\label{sec:th:annulus}

In this section we prove a result that, together with Proposition \ref{pro:main-A} completes the proof of Theorem \ref{th:main}. Note that it does not require the existence of an invariant probability measure of full support.


\begin{theorem}\label{th:annulus} Let $f\colon \T^2\to \T^2$ be a nonwandering homeomorphism homotopic to the identity and $\hat{f}$ a lift of $f$. Suppose that there exists $M>0$ such that 
$$\abs{\smash{P_1(\hat{f}^n(z)-z)}} \leq M \text{ for all } n\in \Z, \, z\in \R^2,$$
and $\rho(\hat{f})$ is not a single point. Then there exists an invariant vertical annulus.
\end{theorem}

We begin with a simple lemma.

\begin{lemma}\label{lem:annulus} Suppose that a homeomorphism $f\colon \T^2\to \T^2$ homotopic to the identity is nonwandering and has an invariant compact set $K$ which is essential but not fully essential. Then $f$ has a periodic essential annulus $A$ which is either invariant or disjoint from its image. In addition, if $A$ is vertical and some lift $\hat{f}$ of $f$ is such that $\rho(\hat{f}) \subset \{0\}\times \R $, then $f(A)=A$.
\end{lemma}
\begin{proof} The hypothesis implies that there exists a connected component $A_0$ of $\T^2\sm K$ which is neither essential nor fully essential. Since the connected components of $\T^2\sm K$ are permuted, $A_0$ is $f^n$-invariant for some $n\in \N$. Moreover, either $n=1$ or $f(A_0)\cap A_0=\emptyset$. Let $A$ be the ``filling'' of $A$, \ie the union of $A_0$ with all inessential connected components of $\T^2\sm A_0$. It is easy to verify that $A$ is homeomorphic to an essential open topological annulus. Moreover, it is still true that $f^n(A)=A$ and either $n=1$ or $f(A)\cap A\neq \emptyset$, so the first claim follows.

To prove the second claim suppose for a contradiction that $n>1$, so that $A$ is disjoint from its image. If $\hat{A}$ is a connected component of $\pi^{-1}(A)$, then $\hat{A} = \hat{A}+(0,1)$ and $\hat{A}$ is disjoint form $\hat{A} + (k,0)$ for any $k\in \Z_*$.  Moreover, there is $k\in \Z$ such that $\hat{f}^n(\hat{A}) = \hat{A}+(k,0)$. Since $P_1(\rho(\hat{f}))=0$ and $P_1(\hat{A})$ is bounded, one easily sees that $k=0$; thus $\hat{f}^n(\hat{A})=\hat{A}$.

On the other hand, since $f(A)\cap A=\emptyset$, one also has $\hat{f}(\hat{A})\cap \hat{A}=\emptyset$. Note that exactly two connected components of $\R^2\sm \hat{A}$ are invariant by the translation $T_2(x,y) = (x,y+1)$. One of these components is unbounded to the right, say $W_+$, and the remaining one, $W_-$, is unbounded to the left. Since $\hat{f}(\hat{A})$ is $T_2$-invariant, we have either $\hat{f}(\hat{A})\subset W_+$ or $\hat{f}(\hat{A})\subset W_-$. But the first case implies that $\hat{f}(W_+)\subset W_+$, while the second case implies that $f(W_-)\subset W_-$. In either case, one concludes $\hat{f}^n(\hat{A})\cap \hat{A}=\emptyset$, a contradiction.
\end{proof}

%
%
%

\begin{proof}[Proof of Theorem \ref{th:annulus}] The proof is somewhat similar to the proof Proposition \ref{pro:main-A}.

Let $H = \{(x,y):x<0\}$, and let $W$ be the unique connected component of $\bigcup_{n\in \Z} \hat{f}^n(H)$ which is unbounded to the left, so that $H\subset W\subset H + (m,0)$ for some $m\in \N$, due to our hypothesis.

Note that $W$ is $\hat{f}$-invariant and $W + (0,1) = W$. This implies that $K=\pi(\bd W)$ is compact and $f$-invariant. Clearly $K$ is essential, since it intersects any horizontal loop in $\T^2$ (indeed, if $\gamma$ is a horizontal loop in $\T^2$ then any connected component of $\pi^{-1}(\gamma)$ is a connected set intersecting both $W$ and its complement, and therefore intersecting $\bd W$).

Suppose that $K$ is not fully essential. Since $K$ is essential and invariant, and clearly under our assumption $P_1(\rho(\hat{f}))=\{0\}$, by Lemma \ref{lem:annulus} there is an invariant essential annulus $A$. Using the assumption that $\rho(\hat{f})$ is not a single point, so that $\rho(\hat{f})=\{0\}\times [a,b]$ for some $a<b$, one easily concludes that $A$ is a vertical invariant annulus, as sought.

It remains to consider the case where $K$ is fully essential. We will show that this case is not possible. To see this, first observe that from Theorem \ref{th:reali-full} and the fact that $\rho(\hat{f})$ necessarily contains different rational points, there are points $z_0, z_1 \in \R^2$ such that $\pi(z_i)\in K$, integers $p_0, p_1$ and positive integers $q_0,q_1$ such that $p_1/q_1\neq p_0/q_0$ and $\hat{f}^{q_i}(z_i) = z_i+p_i$.  We may assume that $z_i\in \bd W=\pi^{-1}(K)$ by using an appropriate integer translation of $z_i$. 

Let $\til{\A} = \R^2/\langle T_2\rangle \simeq \R\times \T^1$, where $T_2(x,y)=(x,y+1)$, and denote by $\tau\colon \R^2\to \til{\A}$ the projection. Let $\til{f}$ be the map induced by $\hat{f}$ on $\til{\A}$, and $\til{T}_1(x,y) = (x+1,y)$ the horizontal translation.

We claim that $\til{f}$ is nonwandering. To see this, suppose that $U$ is a wandering open set of $\til{f}$. By using an appropriate $\til{T}_1$-translation of $U$ and choosing a smaller set, we may assume that $U\subset [0,1]\times \T^1$. Choose an integer $m>2M+1$, and note that $\til{\T}^2 = (\R/m\Z) \times \T^1 \simeq \til{\A}/\langle \til{T}_1^m\rangle$ is a finite covering of $\T^2$, while it is covered by $\til{\A}$ after natural identifications. Let $\sigma\colon \til{\A}\to \til{\T}^2$ be the projection, and let $F$ denote the map of $\til{\T}^2$ lifted by $\til{f}$. Then $F$ is a lift of $f$ to the finite covering $\til{\T}^2$, and since $f$ is nonwandering, By Proposition \ref{pro:nw-lift} we know that $F$ is nonwandering. Note that, since $m>2M+1$, the projection $\sigma$ is injective on $[-M, M+1]\times \T^1$. But our hypothesis on the deviation of $\hat{f}$ implies that $\til{f}^n(U)\subset [-M,M+1]$ for all $n\in \Z$, and so the fact that $\sigma|_{[-M,M+1]}$ is injective implies that 
$$F^n(\sigma(U))\cap \sigma(U)=\sigma(\til{f}^n(U))\cap \sigma(U) = \sigma(\til{f}^n(U)\cap U)=\emptyset$$ for all $n\in \Z$. Hence $\sigma(U)$ is a nonempty wandering open set for $F$, a contradiction.
  
Finally, we observe that $\til{W} = \tau(W)$ is an open subset of $\til{\A}$ such that $(-\infty,0)\times \T^1\subset \til{W}$, and since $W$ is simply connected and $T_2$-invariant, $\til{W}$ is a topological annulus. Thus, as done in \S\ref{sec:cylinder}, if $\til{\A}_*=\til{\A}\sqcup\{L_\infty\}$ is the plane obtained by adding to $\til{\A}$ its end on the left side, $\til{W}_*=\til{W}\sqcup \{L_\infty\}$ is a topological disk which is invariant by the map induced by $\til{f}$ fixing $L_\infty$ (which is still nonwandering), and $\til{z}_i:=\tau(z_i)$ is a periodic point of $\til{f}$ in the boundary of $\til{W}_*$, for $i\in \{0,1\}$. But since the points $\til{z}_i$ have different rotation vectors associated to the lift $\hat{f}$, this contradicts Corollary \ref{coro:kln}, completing the proof.
\end{proof}

\section{Proofs of Theorems \ref{th:main} and \ref{th:main-gen}}

Theorem \ref{th:main} follows immediately from Proposition \ref{pro:main-A} and Theorem \ref{th:annulus}.

\subsection{Proof of Theorem \ref{th:main-gen}}

Let $f$ and $\hat{f}$ be as in the hypotheses of the theorem, and suppose that $\rho(\hat{f})$ is a segment of rational slope containing a point of rational coordinates. By the properties stated in \S\ref{sec:rotation} there is $n\in \N$ such that $\rho(\hat{f}^n)$ contains a point of $\Z^2$, and there is $A\in \GL(2,\Z)$ such that $\rho(A\hat{f}^nA^{-1}) = A\rho(\hat{f})$ is a vertical segment, which also contains a point $v\in \Z^2$. Letting $\hat{g}=A\hat{f}^nA^{-1} - v$, we see that $\hat{g}$ is a lift of $g=A_{\T^2}f^nA_{\T^2}^{-1}$ (where $A_{\T^2}$ is the map induced by $A$ on $\T^2$) and the hypotheses of Theorem \ref{th:main} hold for $g$. Thus there is an invariant vertical annulus $U$ for $g$, which means that $A_{\T^2}U$ is a periodic essential annulus for $f$. 

The converse direction of the theorem is an easy consequence of the definition of rotation set. For completeness, suppose that $A$ is a periodic essential annulus and let $k\in \N$ be such that $f^k(A)=A$. By items (2) and (3) of \cite[Proposition 1.4]{kt-ess}, there is $v\in \Z^2_*$ and some lift $\hat{g}$ of $f^k$ to $\R^2$ such that $\rho(\hat{g})\subset \R v$ (in fact $v$ is the homological direction of $A$). Since $\hat{f}^k$ is another lift of $f^k$, one has that $\hat{g}=\hat{f}^k+w$ for some $w\in \Z^2$, and by the properties from \S\ref{sec:rotation} one has 
$$k\rho(\hat{f}) + w = \rho(\hat{f}^k+w) = \rho(\hat{g}) \subset \R v,$$
which implies that $\rho(\hat{f})\subset \R v - w/k$. Since $\rho(\hat{f})$ is assumed to have more than one point, it must be a segment of rational slope parallel to $v$. Since $\R v - w/k$ contains $w/k\in \Q^2$, it follows that $\R v \cap \Q^2$ is dense in $\R v$, and in particular $\rho(\hat{f})\cap \Q^2\neq \emptyset$.
\qed

\bibliographystyle{koro} 
\bibliography{intervalo}

\end{document}